\documentclass[11pt, reqno]{amsart}

\usepackage{amsmath}
\usepackage{amsthm}
\usepackage{amssymb}
\usepackage{amsfonts}
\usepackage[mathscr]{eucal}
\usepackage{latexsym}
\usepackage{graphicx}
\usepackage{xcolor}
\usepackage{hyperref}

\newtheorem{prop}{Proposition}[section]
\newtheorem{cor}[prop]{Corollary}
\newtheorem{thm}[prop]{Theorem}
\newtheorem{lem}[prop]{Lemma}

\theoremstyle{definition}

\newtheorem{example}[prop]{Example}
\newtheorem{examples}[prop]{Examples}
\newtheorem{rem}[prop]{Remark}
\newtheorem{rems}[prop]{Remarks}
\newtheorem{lab}[prop]{}

\renewcommand{\phi}{\varphi}
\renewcommand{\Re}[1]{\operatorname{Re} #1}

\newcommand{\C}{\mathbb{C}}
\newcommand{\D}{\mathbb{D}}
\newcommand{\N}{\mathbb{N}}
\newcommand{\R}{\mathbb{R}}
\newcommand{\bbS}{\mathbb{S}}
\newcommand{\T}{\mathbb{T}}
\newcommand{\Z}{\mathbb{Z}}

\newcommand{\m}{{\mathfrak{m}}}

\newcommand{\w}{{\mathtt{w}}}
\newcommand{\x}{{\mathtt{x}}}
\newcommand{\y}{{\mathtt{y}}}
\newcommand{\z}{{\mathtt{z}}}

\renewcommand{\iff}{\Leftrightarrow}
\newcommand{\To}{\Rightarrow}
\newcommand{\Ot}{\Leftarrow}

\DeclareMathOperator{\End}{End}
\DeclareMathOperator{\im}{im}
\DeclareMathOperator{\M}{M}
\DeclareMathOperator{\rk}{rk}

\DeclareMathOperator{\supp}{supp}
\DeclareMathOperator{\tr}{tr}

\DeclareTextFontCommand{\textnf}{\normalfont}

\newcommand{\diag}{\mathrm{diag}}
\newcommand{\id}{\mathrm{id}}
\newcommand{\lf}{\mathrm{lf}}

\newcommand{\comp}{\mathbin{\scriptstyle\circ}} 
\renewcommand{\emptyset}{\varnothing}
\renewcommand{\setminus}{\smallsetminus}
\newcommand{\ol}{\overline}
\newcommand{\plus}{{\scriptscriptstyle+}}

\newcommand{\all}{\forall\,}

\renewcommand{\subset}{\subseteq}

\renewcommand{\choose}[2]{\genfrac(){0pt}{}{#1}{#2}}

\newcommand{\bil}[2]{\langle{#1},{#2}\rangle}

\begin{document}

\title
  [Quillen property of real algebraic varieties]
  {Quillen property\\of real algebraic varieties}

\author{Mihai Putinar}
\address{Department of Mathematics\\
         University of California at Santa Barbara\\
         Santa Barbara, California\\
         93106-3080}
\email{mputinar@math.ucsb.edu}

\address{Division of Mathematical Sciences-SPMS\\
            Nanyang Technological University\\
            Singapore 637371}
\email{mputinar@ntu.edu.sg}
\urladdr{https://sites.google.com/site/mihaiputinar/}

\author{Claus Scheiderer}
\address{Fachbereich Mathematik and Statistik\\
         Universit\"at Konstanz  \\
         D-78457 Konstanz, Germany}
\email{claus.scheiderer@uni-konstanz.de}
\urladdr{http://www.math.uni-konstanz.de/\textasciitilde scheider}

\begin{abstract}
A conjugation-invariant ideal $I\subset\C[z_j,\ol z_j\colon
j=1,\dots,n]$ has the Quillen property if every real valued, strictly
positive polynomial on the real zero set $V_\R(I)\subset\C^n$ is a
sum of hermitian squares modulo~$I$. We first relate the Quillen
property to the archimedean property from real algebra. Using
hereditary calculus, we then quantize and show that the Quillen
property implies the subnormality of commuting tuples of Hilbert
space operators satisfying the identities in $I$. In the finite rank
case we give a complete geometric characterization of when the
identities in $I$ imply normality for a commuting tuple of matrices.
This geometric interpretation provides simple means to refute
Quillen's property of an ideal. We also generalize these notions and
results from real algebraic sets to semi-algebraic sets in $\C^n$.
\end{abstract}



\maketitle


\section{Introduction}

On any (affine) real algebraic variety $V$ there exists a natural
source for positivity certificates, namely squares (of regular
functions): Any square, and hence any sum of squares, is nonnegative
whereever it is defined on the $\R$-points of $V$. This observation
lies at the very basis of real algebra, starting with Hilbert's 17th
problem and its solution by Artin. Today the polarity between
positivity and sums of squares is the focus of intense research, both
from theoretical and applied points of view. See \cite{Lt} and
\cite{sch:guide} for recent surveys.

In the present article we consider real algebraic subvarieties $V$ of
complex affine space. The embedding in complex space provides $V$
with additional structure and gives the notion of holomorphic (and
antiholomorphic) elements in the complexified structural rings of
$V$. Accordingly we get a second, more restricted kind of positivity
certificate, namely sums of \emph{hermitian} squares on $V$, that
is, of squared absolute values of holomorphic polynomials restricted
to $V$. Our aim is to study this notion from the points of view of
real algebra, geometry and operator theory.

We work with several complex variables $\z=(z_1,\dots,z_n)$ and their
conjugates $\ol\z=(\ol z_1,\dots,\ol z_n)$. Let $I\subset
\C[\z,\ol\z]$ be a conjugation-invariant ideal, and let $V_\R(I)
\subset\C^n$ be its zero set, a real algebraic subset of $\C^n$. Let
$p\in\C[\z,\ol\z]$ be a conjugation-invariant polynomial that is
non-negative on $V_\R(I)$. We study the question whether $p$ admits
an identity
\begin{equation}\label{hermsos}%
p(\z,\ol\z)\>=\>|h_1(\z)|^2+\cdots+|h_r(\z)|^2+g(\z,\ol\z)
\end{equation}
with $g\in I$, in which $h_1,\dots,h_r\in\C[\z]$ are
\emph{holomorphic} polynomials. When such an identity exists we will
say that $p$ is a \emph{sum of hermitian squares} modulo~$I$.

A classical instance where this property holds is the case of the
unit circle $\T\subset\C$ and its vanishing ideal $I=(z\ol z-1)$.
According to the Riesz-Fej\'er theorem, any $p\in \C[z,\ol z]$
non-negative on $\T$ is a single hermitian square $p=|h(z)|^2$
modulo~$I$.

The first multivariate example with such a property was discovered
almost half a century ago by Quillen \cite{Q}. He studied the unit
sphere $\bbS\subset\C^n$ and its reduced ideal~$I$, and showed that
any $p$ strictly positive on $\bbS$ is a sum of hermitian squares
modulo~$I$.

Quillen's theorem amounts to a Positivstellensatz on the sphere
vis-\`a-vis sums of hermitian squares, rather than ordinary squares.
It is our aim to prove this result in greater generality, and to
study the algebraic and geometric implications of such a result.
Although our approach is basically algebraic, the interlacing with
Hilbert space methods and operator theory is a recurrent theme of
our study.

Fixing a conjugation-invariant ideal $I\subset\C[\z,\ol\z]$, we will
say that $I$ has the Quillen property if the Positivstellensatz holds
for hermitian sums of squares modulo~$I$. Assuming that $V_\R(I)$ is
compact, an abstract characterization of this property comes from
real algebra (Proposition \ref{sohsarchmodi}). This characterization,
however, is often not explicit enough. An improvement, on the
constructive side, is offered by a known link to operator theory.
Specifically, given $p\in\C[\z,\ol\z]$, and given a commuting tuple
$T=(T_1,\dots,T_n)$ of bounded linear operators on a Hilbert space,
define the operator $p(T,T^*)$ using hereditary calculus, thereby
putting all adjoints to the left. We consider the following
properties of the ideal $I$:
\smallskip

(A) (\emph{Archimedean property})
$c-\sum_{j=1}^n|z_j|^2$ is a sum of hermitian squares modulo~$I$, for
some real number~$c$.
\smallskip

(Q) (\emph{Quillen property})
Every conjugation-invariant polynomial strictly positive on $V_\R(I)$
is a sum of hermitian squares modulo~$I$.
\smallskip

(S) (\emph{Subnormality})
Every commuting tuple $T$ of bounded operators on a Hilbert space and
satisfying $f(T,T^*)=0$ for all $f\in I$ is subnormal.
\smallskip

(Sf) (\emph{Finite rank subnormality})
Every commuting tuple $T$ of operators acting on a
\emph{finite-dimensional} Hilbert space and satisfying $f(T,T^*)=0$
for all $f\in I$ is subnormal (hence normal),
\smallskip

(G) (\emph{Geometric normality})
The ideal $I$ is not contained in the any of the ``diamond'' ideals
$$I(a,b)\>=\>\Bigl\{f\in\C[\z,\ol\z]\colon\>f(a,\ol a)=f(a,\ol b)=
f(b,\ol a)=f(b,\ol b)=0\Bigr\}$$
for $a\ne b$ in $\C^n$, and neither in any of their degenerations
$J(a,U)$, see \ref{dfnjau}.
\smallskip

We prove the implications
$$({\rm A})\ \To\ ({\rm Q})\ \To\ ({\rm S})\ \To\ ({\rm Sf})\ \iff\
({\rm G}),$$
$$({\rm A})\ \Ot\ ({\rm Q})\quad\text{if $V_\R(I)$ is compact},$$
see \ref{sohsarchmodi}, \ref{synth} and \ref{umker}. We also analyze
by means of examples why the missing implications do not hold. For
instance, even real conics in $\C$ offer nontrivial features
(\ref{quadpolonevar}, \ref{sfnots}): A~circle satisfies (A), an
eccentric ellipse has property (S) but not (Q), the non-reduced ideal
of a circle with a double point satisfies (Sf) but not (S), and a
hyperbola whose asymptotes are perpendicular doesn't satisfy~(Sf).

We then extend the study of hermitian Positivstellens\"atze from real
algebraic sets to semi-algebraic sets in $\C^n$. To this end we
replace the semiring of hermitian sums of squares mod~$I$ by a
hermitian module $M$, and the real algebraic set $V_\R(I)$ by the
semi-algebraic set $X_M\subset\C^n$ associated with $M$. Defining
properties (Q), (S) and (Sf) for $M$ accordingly, the implications
(Q) $\To$ (S) $\To$ (Sf) remain true. When $M$ is archimedean and
satisfies a polynomial convexity property, the reverse (S) $\To$ (Q)
holds true as well (Theorem \ref{conversepolyconv}). When $M$ is
finitely generated, we prove that the Quillen property is
incompatible with $X_M$ containing an analytic disc (Theorem
\ref{ausbau}). In this direction we mention article \cite{dAP2},
where a notion of hermitian complexity was introduced for
conjugation-invariant ideals with the precise aim of bridging the gap
between Quillen's property at one end and the existence of analytic
discs in the support at the other.

At the end of the paper we make a few historical comments putting
this work into perspective, mentioning some of the analytic roots and
applications of hermitian sums of squares.


\section{Preliminaries and notation}

\begin{lab}\label{hermsetup}%
Let $\C[\z,\ol\z]$ be the polynomial ring in $2n$ independent
variables $\z=(z_1,\dots,z_n)$ and $\ol\z=(\ol z_1,\dots,\ol z_n)$.
On $\C[\z,\ol\z]$ we consider the $\C/\R$-involution $z_j^*=\ol z_j$
($j=1,\dots,n$). Thus
$$\Bigl(\sum_{\alpha,\beta}a_{\alpha,\beta}\,\z^\alpha\,\ol\z^\beta
\Bigr)^*\>=\>\sum_{\alpha,\beta}\ol{a_{\alpha,\beta}}\,\z^\beta\,
\ol\z^\alpha$$
for $a_{\alpha,\beta}\in\C$, with the usual multi-index notation
$\z^\alpha\ol\z^\beta=\prod_{j=1}^nz_j^{\alpha_j}\ol z_j^{\beta_j}$.
The fixed ring of $*$ is the polynomial ring $\R[\x,\y]$ generated by
the $2n$ variables $\x=(x_1,\dots,x_n)$ and $\y=(y_1,\dots,y_n)$,
where $x_j=\frac12(z_j+\ol z_j)$ and $y_j=\frac1{2i}(z_j-\ol z_j)$
(and $i=\sqrt{-1}$). Thus $\C[\z,\ol\z]$ is identified with
$\R[\x,\y] \otimes\C$, and under this identification, the involution
$*$ becomes complex conjugation in the second tensor component.

Given $f\in\C[\z,\ol\z]$ and $a,b\in\C^n$, we write $f(a,b)\in\C$ for
the result of substituting $a$ for $\z$ and $b$ for $\ol\z$. We often
abbreviate $f(a):=f(a,\ol a)$.
\end{lab}

\begin{lab}\label{dictideals}%
There is a one-to-one correspondence between $*$-invariant ideals $J$
of $\C[\z,\ol\z]$ and arbitrary ideals $I$ of $\R[\x,\y]$, given by
$J\mapsto I:=J\cap\R[\x,\y]$.
Given an ideal $I$ of $\R[\x,\y]$, we denote the zero set of $I$ in
$\C^n$ by
$$V_\R(I)\>:=\>\{a\in\C^n\colon\all f\in I\ f(a)=0\}.$$
This is a real algebraic subset of $\C^n$.
\end{lab}

\begin{lab}\label{dfnhermsos}%
For every $p\in\C[\z,\ol\z]$, the hermitian norm $|p|^2:=pp^*$ is a
sum of two (usual) squares in $\R[\x,\y]$. The convex cone in
$\R[\x,\y]$ generated by $\{|p|^2\colon p\in\C[\z,\ol\z]\}$ will be
denoted by $\Sigma$; it is the cone of all (usual) \emph{sums of
squares} in $\R[\x,\y]$. The smaller convex cone in $\R[\x,\y]$
generated by $\{|p|^2\colon p\in\C[\z]\}$ is denoted by $\Sigma_h$.
Its elements are called the \emph{hermitian sums of squares}.
\end{lab}

\begin{lab}\label{dfnarch}%
We recall a few notions from real algebra. Given an $\R$-algebra $A$
(i.e., a commutative ring containing~$\R$), a subset $S\subset A$
will be called a \emph{semiring} in $A$ if $S$ contains the
nonnegative real numbers and is closed in $A$ under taking sums and
products.
Given a semiring $S$, an \emph{$S$-module} is a subset $M$ of $A$
with $M+M\subset M$, $SM\subset M$ and $1\in M$. A~particularly
important semiring is $\Sigma A^2$, the set of all (finite) sums of
squares in $A$. The modules over this semiring are usually referred
to as the quadratic modules in $A$.

The $S$-module $M$ is said to be \emph{archimedean} if $A=\R+M$, that
is, if for every $f\in A$ there exists $c\in\R$ with $c\pm f\in M$.

In this paper we will mostly be concerned with the $\R$-algebra
$A=\R[\x,\y]$ and with the two semirings $\Sigma_h\subset\Sigma$ in
$\R[\x,\y]$.
\end{lab}

\begin{lab}
Given a module $M$ over some semiring $S$ in $\R[\x,\y]$, we write
$$X_M\>:=\>\bigl\{a\in\C^n\colon\>\all g\in M\ g(a)\ge0\bigr\},$$
which is a closed subset of $\C^n$.
\end{lab}

The celebrated archimedean Positivstellensatz from real algebra
(see \cite{PD} or \cite{sch:guide}) implies:

\begin{thm}\label{archpss}%
If $M$ is a module over an archimedean semiring $S$ in $\R[\x,\y]$,
then $M$ contains any $f\in\R[\x,\y]$ that is strictly positive on
the set $X_M$.
\end{thm}


\section{Hermitian sums of squares and
  subnormal tuples of operators}\label{sect:hermsossubnop}%

\begin{lab}\label{dfnpropa}%
Let $\R[\x,\y]\subset\C[\z,\ol\z]$ be the fixed ring of $*$ (see
\ref{hermsetup}). Given any ideal $I\subset\R[\x,\y]$, we will
consider the semiring $S=\Sigma_h+I$ in $\R[\x,\y]$. Note that $X_S=
V_\R(I)$. Consider the following two properties of the ideal~$I$:
\begin{itemize}
\item[\rm(A)]
\emph{(Archimedean Property)}
The semiring $\Sigma_h+I$ in $\R[\x,\y]$ is archi\-medean (see
\ref{dfnarch});
\item[\rm(Q)]
\emph{(Quillen Property)}
$\Sigma_h+I$ contains every $f\in\R[\x,\y]$ that is strictly positive
on $V_\R(I)$.
\end{itemize}
We'll also refer to (A) by saying that $\Sigma_h$ is archimedean
modulo~$I$.

Given any $*$-invariant ideal $J\subset\C[\z,\ol\z]$, we will say that
$J$ has property (A) resp.\ (Q) if the ideal $J\cap\R[\x,\y]$ of
$\R[\x,\y]$ has the respective property (see \ref{dictideals}).

The following result was proved in \cite{PS} (Theorem 2.1 and
Proposition 2.2). It is essentially an application of the archimedean
Positivstellensatz \ref{archpss}:
\end{lab}

\begin{prop}\label{sohsarchmodi}%
For any ideal $I\subset\R[\x,\y]$, the following are equivalent:
\begin{itemize}
\item[(i)]
$I$ has the Archimedean property $(\rm A)$;
\item[(ii)]
$I$ has the Quillen property $(\rm Q)$, and $V_\R(I)$ is compact;
\item[(iii)]
$I$ contains a polynomial of the form $||\z||^2+p+a$, where $p\in
\Sigma_h$ and $a\in\R$.
\end{itemize}
(We are using the shorthand $||\z||^2:=|z_1|^2+\cdots+|z_n|^2$.)
\end{prop}

\begin{rems}
\hfil
\smallskip

1.\
Quillen's theorem \cite{Q}, reproved later by Catlin-D'Angelo
\cite{CdA1}, was mentioned in the introduction. The statement is
recovered here in a purely algebraic way, as a very particular
instance of Proposition \ref{sohsarchmodi}.

As observed in \cite{CdA1}, Quillen's theorem implies the following
classical theorem due to P\'olya: Given a homogeneous polynomial
$f\in\R[x_1,\dots,x_n]$ strictly positive on $\{a\in\R^n\colon
a_1\ge0,\dots,a_n\ge0\}\setminus\{(0,\dots,0)\}$, the form
$(x_1+\cdots+x_n)^Nf$ has positive coefficients for large enough
$N\ge0$.
\smallskip

2.\
Condition (iii) of \ref{sohsarchmodi} gives an abstract algebraic
characterization of the ideals $I$ with $V_\R(I)$ compact and with
property~(Q). Note that the Positivstellensatz for usual sums of
squares holds whenever $V_\R(I)$ is compact, by Schm\"udgen's theorem
\cite{Sm}. In contrast, ``most'' ideals with
$V_\R(I)$ compact do not satisfy property (Q) (see, e.g.,
\ref{quadpolonevar} below).

The applicability of \ref{sohsarchmodi}(iii) as an algebraic
criterion for property (Q) is somewhat limited, since this condition
is not sufficiently explicit. In particular, it is usually cumbersome
to prove that an ideal $I$ does \emph{not} contain any polynomial of
the form given in (iii).
Therefore it is desirable to know other conditions on $I$ that are
necessary for (A) resp.~(Q), and that are more easily checked. In
this section and the next we will offer two conditions of very
different nature that are both necessary for the Quillen property,
one operator-theoretic and one ideal-theoretic.
\smallskip

3.\
Part of the original motivation for this work came from a question of
D'Angelo. Given a compact real algebraic set $X\subset\C^n$ which is
the boundary of a strictly pseudo-convex region in $\C^n$, D'Angelo
had asked whether every strictly positive polynomial on $X$ is a
sum of hermitian squares on $X$. This question was answered in the
negative, see \cite{PS}.
\end{rems}
\pagebreak

\begin{examples}\label{roleofcpt}%
\hfil
\smallskip

1.\
The Quillen property (Q) alone does not imply the archimedean
property (A), since $V_\R(I)$ need not be compact. This is seen by
considering a line in $\C$, given (say) by the ideal $I=(y)\subset
\R[x,y]$. Condition (Q) is satisfied since, in fact, $\Sigma_h+I$
contains every $f\in\R[x,y]$ nonnegative on the line $y=0$. Indeed,
such $f$ is a sum of two usual squares modulo~$I$, from which one
sees easily that $f$ is congruent modulo~$I$ to a single hermitian
square, i.e., $f\equiv|p|^2$ (mod~$I$) with $p\in\C[z]$.
\smallskip

2.\
If $n=1$ and $f\in\R[x,y]$ has degree~$2$, the principal ideal
$I=(f)$ satisfies the Archimedean property (A) if and only if there
exist $\alpha\in\C$ and $a,\,c\in\R$ with $f=a|z-\alpha|^2+c$. This
will be proved in Theorem \ref{quadpolonevar} below.
\end{examples}

\begin{lab}\label{heredcalc}%
Let $E$ be a (separable complex) Hilbert space, and let $B(E)$ denote
the algebra of bounded linear operators on $E$. Fix a tuple $T=
(T_1,\dots,T_n)$ of operators $T_j\in B(E)$ that commute pairwise. We
use hereditary calculus (see \cite{AM} Section 14.2 for more
details). Given a monomial $f=\z^\alpha\ol\z^\beta$ (with $\alpha,\,
\beta\in\Z_\plus^n$) we write
$$f(T,T^*)\>:=\>T^{*\beta}T^\alpha.$$
We extend this definition $\C$-linearly, thereby putting all adjoints
to the left. This defines the $\C$-linear map
$$\psi_T\colon\>\C[\z,\ol\z]\to B(E),\quad f(\z,\ol\z)\>\mapsto\>
\psi_T(f)=f(T,T^*).$$
The map $\psi_T$ commutes with the involution, i.e.\ $\psi_T(f^*)=
\psi_T(f)^*$. In particular, $\psi_T(f)$ is self-adjoint for $f=f^*$.
Note that
$$\psi_T\Bigl(\ol{q(\z)}\cdot f(\z,\ol\z)\cdot p(\z)\Bigr)\>=\>
\psi_T(q)^*\,\psi_T(f)\,\psi_T(p)$$
for $p,\,q\in\C[\z]$ and $f\in\C[\z,\ol\z]$. The set
$$M_T\>:=\>\bigl\{f\in\R[\x,\y]\colon\psi_T(f)\ge0\bigr\}$$
(of real polynomials $f$ for which the self-adjoint operator $\psi_T
(f)$ is nonnegative) is a $\Sigma_h$-module in $\R[\x,\y]$, since
$\psi_T(|p|^2f)=\psi_T(\ol{p(\z)}f(\z,\ol\z)p(\z))=p(T)^*\,\psi_T(f)
\,p(T)$ holds for $f\in\C[\z,\ol\z]$ und $p\in\C[\z]$. The support
$M_T\cap(-M_T)=\ker(\psi_T)$ of $M_T$ is an ideal in $\R[\x,\y]$.
Note that the subset $M_T$ of $\R[\x,\y]$ is closed with respect to
the finest locally convex topology on $\R[\x,\y]$.
\end{lab}

\begin{lab}
Recall that the tuple $T$ is said to be (jointly) \emph{subnormal} if
$T$ can be extended to a commuting tuple of normal operators on a
larger Hilbert space, i.e., if there is a tuple $T'=(T'_1,\dots,
T'_n)$ of commuting normal operators on a Hilbert space $E'$ such
that $E'$ contains $E$ and the $T'_i$ leave $E$ invariant and satisfy
$T'_i|_E=T_i$ for $i=1,\dots,n$.
Note that subnormal is equivalent to normal when $\dim(E)<\infty$.
For details see~\cite{Co}.

According to the Halmos-Bram-It\^o criterion (see \cite{It}),
the commuting tuple $T=(T_1,\dots,T_n)$ is subnormal if and only if
$$\sum_{\alpha,\beta}\bil{T^\alpha\xi_\beta}{T^\beta\xi_\alpha}\>\ge
\>0$$
for all finitely supported families $\{\xi_\alpha\}_{\alpha\in
\Z_\plus^n}$ in $E$. Using this criterion we show:
\end{lab}

\begin{prop}\label{sigmasubn}%
Let $T=(T_1,\dots,T_n)$ be a commuting tuple in $B(E)$. Then $T$ is
subnormal if and only if\/ $\Sigma\subset M_T$.
\end{prop}

In other words, the tuple $T$ is subnormal if and only if $\psi_T
(|p|^2)\ge0$ holds for every $p\in\C[\z,\ol\z]$.

\begin{proof}
Assume $\psi_T(|p|^2)\ge0$ for every $p\in\C[\z,\ol\z]$. To prove
that $T$ is subnormal we can, using a result of Stochel (\cite{St},
Cor.\ 3.2),
assume that there exists a cyclic vector $\xi$ for $T$, i.e.\ the
linear span of $\{T^\alpha\xi\colon\alpha\in\Z_\plus^n\}$ is dense in
$E$. It suffices to verify the Halmos-Bram-It\^o condition for all
finite families $\{\xi_\alpha\}$ lying in the linear span of
$\{T^\alpha\xi\colon\alpha\in\Z_\plus^n\}$. So let $\xi_\alpha=
p_\alpha(T)\xi$ where $p_\alpha\in\C[\z]$ for $\alpha\in\Z_\plus^n$
(and $p_\alpha=0$ for almost all $\alpha$), and consider $p:=\sum
_\alpha p_\alpha(\z)\,\ol\z^\alpha\,\in\C[\z,\ol\z]$. Since
$$|p|^2\>=\>\sum_{\alpha,\beta}p_\alpha(\z)\,\ol{p_\beta(\z)}\,
\z^\beta\,\ol\z^\alpha,$$
the assumption $\Sigma\subset M_T$ gives
$$0\le\bil{\psi_T(|p|^2)\xi}\xi\>=\>\sum_{\alpha,\beta}\bigl\langle
T^\beta p_\alpha(T)\xi,\>T^\alpha p_\beta(T)\xi\bigr\rangle\>=\>\sum_
{\alpha,\beta}\bigl\langle T^\beta\xi_\alpha,\>T^\alpha\xi_\beta
\bigr\rangle,$$
which shows that $T$ is subnormal. Conversely, the same argument
shows that $T$ subnormal implies $\Sigma\subset M_T$.
\end{proof}

\begin{lab}\label{dfnssf}%
We shall consider the following properties of an ideal $I\subset
\R[\x,\y]$:
\begin{itemize}
\item[(\rm S)]
\emph{(Subnormality)}
Every commuting tuple $T=(T_1,\dots,T_n)$ of bounded linear operators
in a Hilbert space satisfying $p(T,T^*)=0$ for every $p\in I$ is
subnormal.
\item[(\rm Sf)]
\emph{(Finite rank subnormality)}
Every commuting tuple $T=(T_1,\dots,T_n)$ of complex matrices
satisfying $p(T,T^*)=0$ for every $p\in I$ is normal.
\end{itemize}
\end{lab}

Trivially (S) implies (Sf). Condition (Sf) will be considered in the
next section. Here we first show that condition (S) is necessary
for the Quillen property (Q). (This fact was announced without proof
in \cite{PSell} Corollary 2.2).

\begin{prop}\label{synth}%
For any ideal $I\subset\R[\x,\y]$, Quillen property $(\rm Q)$
implies the subnormality condition $(\rm S)$.
\end{prop}

\begin{proof}
Assume (Q) holds for $I$. Given a commuting tuple $T$ of bounded
operators with $I\subset\ker(\psi_T)$, we have $\Sigma_h+I\subset
M_T$. Since $M_T$ is closed with respect to the finest locally convex
topology of $\R[\x,\y]$ (\ref{heredcalc}), it follows from (Q) that
$M_T$ contains every polynomial that is nonnegative on $V_\R(I)$. In
particular we have $\Sigma\subset M_T$, which implies that $T$ is
subnormal (Proposition \ref{sigmasubn}).
\end{proof}

\begin{rem}
In the case when $V_\R(I)$ is compact, we can give a very short proof
of Proposition \ref{synth}, using Athavale's theorem \cite{At}.
Indeed, assume that $V_\R(I)$ is compact and (Q) holds for $I$. After
suitably scaling the variables we can assume $|\xi_j|<1$ for every
$\xi=(\xi_1,\dots,\xi_n)\in V_\R(I)$. Let $T$ be a commuting tuple of
bounded operators satisfying $I\subset\ker(\psi_T)$. In order to show
that $T$ is subnormal it suffices, by \cite{At} Theorem 4.1, to show
for any tuple $\alpha=(\alpha_1,\dots,\alpha_n)$ of nonnegative
integers that
$$f\>:=\>\prod_{j=1}^n(1-|z_j|^2)^{\alpha_j}\>\in\>M_T.$$
Now $f>0$ on $V_\R(I)$, so the assumption on $I$ implies $f\in
\Sigma_h+I$, from which $\psi_T(f)\ge0$ is obvious.
\end{rem}

\begin{rem}\label{converse}%
The subnormality property (S) on an ideal $I\subset\R[\x,\y]$ is
strictly weaker than the Quillen property (Q). An immediate example
to show this is given by the ideal $I=(x_1,\dots,x_n)=(z_j+\ol z_j
\colon j=1,\dots,n)$ in $\R[\x,\y]$: Every commuting tuple $T$ of
operators with $I\subset\ker(\psi_T)$ consists clearly of normal
operators.
On the other hand, for any $n\ge2$ there exist strictly positive
polynomials on $V_\R(I)\cong\R^n$ that are not even sums of usual
squares, and \emph{a~fortiori} not of hermitian squares. For
instance, adding a positive constant to the well-known Motzkin
polynomial $y_1^4y_2^2+y_1^2y_2^4-3y_1^2y_2^2+1$ gives such an
example.

It is less straightforward to find an ideal $I$ satisfying (S) but
not (Q), for which $V_\R(I)$ is compact. Let $f(z,\ol z)=0$ be the
equation of an ellipse that is not a circle. Then every bounded
operator $T$ satisfying $f(T,T^*)=0$ is subnormal, that is, the
principal ideal $I=(f)$ satisfies (S). But $I$ does not have the
Quillen property, see Theorem \ref{quadpolonevar} below, and also
\cite{PSell}.
\end{rem}


\section{Normal tuples of matrices}

In this section we will provide a complete geometric characterization
of the ideals satisfying finite rank subnormality (Sf).
More specifically, we will explicitly list those ideals that are
maximal with respect to not satisfying~(Sf).

First we need some preparation. It seems more natural here to work
with $*$-invariant ideals of $\C[\z,\ol\z]$, rather than with ideals
of $\R[\x,\y]$.

\begin{lab}\label{dfniab}%
Given $a\ne b$ in $\C^n$, let $I(a,b)\subset\C[\z,\ol\z]$ be the
ideal consisting of all polynomials $f(\z,\ol\z)$ with
$$f(a,\ol a)\>=\>f(b,\ol b)\>=\>f(a,\ol b)\>=\>f(b,\ol a)\>=\>0.$$
Clearly, $I(a,b)$ is $*$-invariant. As an ideal in $\C[\z,\ol\z]$,
note that $I(a,b)$ is generated by the polynomials $p(\z)$ and
$p(\z)^*$, where $p(\z)\in\C[\z]$ is a holomorphic polynomial
satisfying $p(a)=p(b)=0$.

These ideals were introduced in \cite{PS}, where $I(a,b)$ was denoted
by $J_{a,b}$.
\end{lab}

\begin{lab}\label{dfnjau}%
The usual inner product on the space of hermitian $n\times n$
matrices will be denoted by $\bil ST:=\tr(ST)$. Given $a\in\C^n$ and
a
complex hermitian $n\times n$ matrix $U\ne0$, let $J(a,U)$ be the set
of all $f\in\C[\z,\ol\z]$ such that
\begin{itemize}
\itemsep3pt
\item[(1)]
$f(a,\ol a)=0$,
\item[(2)]
$U\cdot\nabla_\z f(a,\ol a)\>=\>\ol U\cdot\nabla_{\ol\z}f(a,\ol a)
\>=\>0$,
\item[(3)]
$\bigl\langle U,\>\nabla^2_{\z\ol\z}f(a,\ol a)\bigr\rangle\>=\>0$.
\end{itemize}
Here we denote the holomorphic resp.\ antiholomorphic gradient by
$$\nabla_\z f=\Bigl(\frac{\partial f}{\partial z_j}\Bigr)_{j=1,\dots,
n},\quad\nabla_{\ol\z}f=\Bigl(\frac{\partial f}{\partial\ol z_j}
\Bigr)_{j=1,\dots,n}$$
(regarded as column vectors), and the mixed Hessian (Levi form) by
$$\nabla^2_{\z\ol\z}f=\Bigl(\frac{\partial^2f}{\partial z_j\,\partial
\ol z_k}\Bigr)_{j,k=1,\dots,n}$$
It is easy to see that $J(a,U)$ is a $*$-invariant ideal in $\C[\z,
\ol\z]$.
\end{lab}

\begin{example}\label{keyex}%
With a view toward the proof of Theorem \ref{umker} below, let us
consider the following example. Fix an integer $r\ge1$ and column
vectors $w_1,\dots,w_n\in\C^r$, not all of them zero. Moreover, let
$a=(a_1,\dots,a_n)\in\C^n$, and let
$$T_j\ =\ \begin{pmatrix}a_j&0\\w_j&a_jI_r\end{pmatrix}\ \in\>
\M_{r+1}(\C)$$
(we are using a $(1,r)$ block matrix notation).
Clearly, $T=(T_1,\dots,T_n)$ is a commuting tuple of matrices, and is
not normal since $w_j\ne0$ for at least one~$j$. A straightforward
calculation shows $\ker(\psi_T)=J(a,U)$, where $U$ is the nonnegative
hermitian $n\times n$-matrix
$$U\>=\>\bigl(w_j^*w_k\bigr)_{1\le j,k\le n}.$$
Note that the rank of $U$ is the dimension of the linear span of
$w_1,\dots,w_n$ in $\C^r$.
\end{example}

Next comes the main result of this section. It gives a complete
ideal-theoretic characterization of condition~(Sf):

\begin{thm}\label{umker}%
Let $I\subset\C[\z,\ol\z]$ be an ideal. The following are equivalent:
\begin{itemize}
\item[$(\rm Sf)$]
Every commuting tuple $T=(T_1,\dots,T_n)$ of complex matrices
satisfying $I\subset\ker(\psi_T)$ is normal;
\item[(\rm G)]
$I$ is not contained in $I(a,b)$ for any pair $a\ne b$ in $\C^n$, and
neither in $J(a,U)$ for any $a\in\C^n$ and any nonnegative hermitian
$n\times n$ matrix $U\ne0$.
\end{itemize}
\end{thm}

\begin{lab}
We prove the implication (Sf) $\To$ (G) by contraposition. More
precisely, we will show:
\begin{itemize}
\item[(a)]
For any $a\ne b$ in $\C^n$, there exists a commuting non-normal
$n$-tuple $T$ of $2\times2$ matrices with $\ker(\psi_T)=I(a,b)$.
\item[(b)]
For any $a\in\C^n$ and any nonnegative hermitian $n\times n$ matrix
$U\ne0$, there exists a commuting non-normal $n$-tuple $T$ of
$m\times m$ matrices with $\ker(\psi_T)=J(a,U)$. (We can take
$m=\rk(U)+1$ here.)
\end{itemize}
In fact, (b) has already been proved by Example \ref{keyex}. (The
last assertion comes from the fact that a nonnegative hermitian
matrix $U$ of rank $r\ge1$ can be written $U=W^*W$ with $W\in
\M_{r\times n}(\C)$.)
Assertion (a) will be proved in \ref{2fTo3a} and \ref{nml2x2}.
The reverse implication (G) $\To$ (Sf) will be proved in~\ref{3To2f}.
\end{lab}

\begin{lab}\label{2fTo3a}%
Let $a\ne b$ in $\C^n$. Fix two linearly independent vectors $u,\,v$
in $\C^2$ that are not perpendicular. Let $T_j\in M_2(\C)$ be the
matrix satisfying $T_ju=a_ju$ and $T_jv=b_jv$ ($j=1,\dots,n$). Then
$T=(T_1,\dots,T_n)$ is a commuting tuple of matrices. Clearly, the
matrix $T_j$ fails to be normal for any index $j$ with $a_j\ne b_j$,
and in particular, the tuple $T$ is not normal. We claim
$\ker(\psi_T)=I(a,b)$. The inclusion $I(a,b)\subset\ker(\psi_T)$ is
obvious.
Since the vector space $\C[\z,\ol\z]/I(a,b)$ has dimension~$4$, we
have to show that the linear map $\psi_T\colon\C[\z,\ol\z]\to
M_2(\C)$ is surjective.
This in turn follows immediately from the
following lemma.
\end{lab}

\begin{lem}\label{nml2x2}%
Let $S\in\M_2(\C)$. The matrices $I$, $S$, $S^*$, $S^*S$ are linearly
dependent if and only if $S$ is normal.
\end{lem}

\begin{proof}
Let $W_S$ be the linear span of $I$, $S$, $S^*$ and $S^*S$ in $\M_2
(\C)$. We have $W_S=W_{S-\lambda I}$ for every $\lambda\in\C$.
Since $S-\lambda I$ is normal iff $S$ is normal, we can replace $S$
by $S-\lambda I$ for any $\lambda\in\C$. In particular, we may do
this for $\lambda$ an eigenvalue of $S$. After changing to a suitable
orthonormal basis we can therefore assume $S=\choose{0\ a}{0\ b}$
where $a$, $b\in\C$. For this matrix it is immediate that $W_S\ne\M_2
(\C)$ if and only if $a=0$, if and only if $S$ is normal.
\end{proof}

\begin{lab}\label{3To2f}%
We now show that (G) implies (Sf) in Theorem \ref{umker}, again by
contraposition. To this end let $E$ be a finite-dimensional Hilbert
space, and let $T=(T_1,\dots,T_n)$ be a commuting tuple of
endomorphisms of $E$ such that at least one $T_j$ is not normal.
We'll show that the ideal $\ker(\psi_T)$ of $\C[\z,\ol\z]$ is
contained in one of the ideals $I(a,b)$ or $J(a,U)$, as in~(G).

Let $F$ be any $T$-invariant subspace of $E$ (that is, $T_jF\subset
F$ holds for each~$j$), and let $T|F$ denote the restriction of $T$
to $F$. So $T|F$ is a commuting tuple of endomorphisms of $F$. Let
$i\colon F\to E$ be the inclusion map and $\pi\colon E\to F$ the
orthogonal projection onto $F$, and let
$$\rho\colon\End(E)\to\End(F),\quad\rho(S)=\pi\comp S\comp i.$$
For $S\in\End(E)$ we have $(S|F)^*=\rho(S^*)$. Moreover, if $S$
leaves $F$ invariant, then $\rho(S'S)=\rho(S')\rho(S)$, $\rho(S^*S')=
\rho(S)^*\rho(S')$ hold for any $S'\in\End(E)$.
As maps $\C[\z,\ol\z]\to\End(F)$, we therefore have $\psi_{T|F}=
\rho\comp\psi_T$.
In particular, $\ker(\psi_T)\subset\ker(\psi_{T|F})$. In order to
prove what we want, we can therefore replace $E$ and $T$ by $F$ and
$T|F$ whenever $F$ is a $T$-invariant subspace of $E$ for which $T|F$
is not normal.

For any tuple $a=(a_1,\dots,a_n)\in\C^n$, denote by
$$E(T,a)\>=\>\{\xi\in E\colon\>(T_j-a_j)\xi=0\text{ \ for \ }j=1,
\dots,n\}$$
resp.\ by
$$E_\infty(T,a)\>=\>\{\xi\in E\colon\>(T_j-a_j)^{\dim(E)}\xi=0
\text{ \ for \ }j=1,\dots,n\}$$
the $a$-eigenspace resp.\ the generalized $a$-eigenspace of $T$.
These are $T$-invariant subspaces of $E$, and $E=\bigoplus_{a\in\C^n}
E_\infty(T,a)$.
Since $T$ is not normal, one of the following two situations occurs:
\begin{itemize}
\item[(1)]
One of the $T_j$ is not diagonalizable;
\item[(2)]
each $T_j$ is diagonalizable, but for at least one index $j$ there
are two eigenspaces of $T_j$ that are not perpendicular.
\end{itemize}
Let us first discuss case (2). By assumption we have $E=\bigoplus
_{a\in\C^n}E(T,a)$, and there exist $a\ne b$ in $\C^n$ such that
$E(T,a)$ and $E(T,b)$ are not perpendicular.
Pick vectors $x\in E(T,a)$ and $y\in E(T,b)$ that are not
perpendicular. The two-dimensional subspace $F$ spanned by $x$ and
$y$ is $T$-invariant, and $T|F$ is not normal.
By the argument used in \ref{2fTo3a}, we see that $\ker(\psi_{T|F})=
I(a,b)$. So we are finished with case~(2).

Now we discuss case~(1) and assume that one of the $T_j$ cannot be
diagonalized. Then there exists $a\in\C^n$ with $E(T,a)\ne E_\infty
(T,a)$.
Replacing $E$ by $E_\infty(T,a)$ and $T_j$ by $T_j-a_j$ for
each $j$ (the latter corresponding to a change of variables $z_j\to
z_j-a_j$ in the polynomial ring), we can assume that each $T_j$ is
nilpotent and $T_j\ne0$ for at least one~$j$. Let $c\ge2$ be the
highest order of nilpotency among the $T_j$, that is, assume $T_j^c
=0$ for all $j$ and $T_{j_0}^{c-1}\ne0$ for one index $j_0$.
Replacing $E$ by $\ker(T_{j_0}^{c-2})$ we can assume $T_j^2=0$ for
all~$j$.

Let $V_j=\ker(T_j)$ for $j=1,\dots,n$. Whenever there are two indices
$j,\>k$ with $V_j\not\subset V_k$, we can replace $E$ by $V_j$.
Iterating this step we arrive at the case where all nonzero operators
among $T_1,\dots,T_n$ have the same kernel $V\ne E$.
Thus, for each $j$, we have either $T_j=0$ or $\im(T_j)\subset\ker
(T_j)=V$, and the latter occurs for at least one index~$j$.

Choose a nonzero vector $x\in V^\bot$. The subspace $F:=\C x\oplus
V$ of $E$ is $T$-invariant, and we can replace $E$ with $F$.
Put $y_j=T_jx$ ($j=1,\dots,n$), and let $W\subset V$ be the linear
span of $y_1,\dots,y_n$. We can replace $E$ by $\C x\oplus W$, and
have now arrived at a minimal non-normal tuple of operators.

Let $r=\dim(W)$, so $1\le r\le n$. Fixing an orthonormal linear basis
of $W$, we represent the operators $T_j$ by $(r+1)\times(r+1)$
matrices as
$$T_j\>=\>\begin{pmatrix}0&0&\cdots&0\\y_{1j}&0&\cdots&0\\
\vdots&\vdots&&\vdots\\y_{rj}&0&\cdots&0\end{pmatrix}.$$
Let $w_j=(y_{1j},\dots,y_{rj})$, regarded as a column vector ($j=1,
\dots,n$), and let
$$U\>=\>\bigl(w_j^*w_k\bigr)_{1\le j,k\le n}$$
a psd hermitian matrix of rank~$r$. From Example \ref{keyex} we see
$\ker(\psi_T)=J(a,U)$.

This completes the proof of Theorem \ref{umker}.
\qed
\end{lab}

\begin{rem}
Ideals of the form $I(a,b)$ or $J(a,U)$, as in \ref{umker}, are
pairwise incomparable with respect to inclusion, except that $J(a,U)=
J(a,cU)$ for every real number $c>0$. To see that $J(a,U)\subset
J(a,U')$ implies $U'=cU$ with $c>0$, observe that the mixed Hessians
of elements of $J(a,U)$ are precisely the matrices that are
orthogonal to $U$ (condition (3) of \ref{dfnjau}). Therefore $U$ is
determined by $J(a,U)$ up to (positive) scaling.

So we see that the ideals $I(a,b)$ and $J(a,U)$ are precisely the
maximal ones among the ideals of relations between non-normal
commuting tuples of matrices (in the sense of
hereditary calculus).
\end{rem}

\begin{rem}
A complex square matrix may be non-normal for two reasons: It may
fail to be diagonalizable, or it may have two non-perpendicular
eigenvectors for different eigenvalues. The ideals $J(a,U)$ and
$I(a,b)$ in Theorem \ref{umker} correspond to these two
possibilities. More precisely, if $T=(T_1,\dots,T_n)$ is a commuting
tuple of matrices, and if one of the $T_j$ is not diagonalizable,
then $\ker(\psi_T)\subset J(a,U)$ for some pair $(a,U)$. On the other
hand, if the $T_j$ are diagonalizable but one of them has two
non-perpendicular eigenspaces, then $\ker(\psi_T)\subset I(a,b)$ for
some pair $(a,b)$. Both assertions are clear from the proof in
\ref{3To2f}.
\end{rem}

\begin{rem}
Consider commuting tuples $T=(T_1,\dots,T_n)$ in $B(E)$ where $E$ is
a complex Hilbert space, together with the associated maps $\psi_T
\colon\C[\z,\ol\z]\to B(E)$ given by hereditary calculus
(\ref{heredcalc}). It is a consequence of Fuglede's theorem that the
tuple $T$ is normal if and only if $\psi_T$ is a ring homomorphism.
As a consequence of Theorem \ref{umker}, we can add another
characterization, as long as $E$ has finite dimension. It shows that
the normality of a commuting tuple $T$ of matrices can be decided
from its ideal $\ker(\psi_T)$ of relations:
\end{rem}

\begin{cor}
A commuting tuple $T$ of matrices is normal if and only if the ideal
$\ker(\psi_T)$ is not contained in $I(a,b)$ for any $a\ne b$ in
$\C^n$, and neither in $J(a,U)$ for any $a\in\C^n$ and any
nonnegative hermitian $n\times n$-matrix $U\ne0$.
\end{cor}

\begin{proof}
Indeed, if $T$ is normal, there is an orthogonal basis of
simultaneous eigenvectors. This implies that $\ker(\psi_T)$ is an
intersection of finitely many ideals $\m_a=\{f\in\C[\z,\ol\z]\colon
f(a,\ol a)=0\}$, $a\in\C^n$. Such an intersection is never contained
in any of the ideals $I(a,b)$ or $J(a,U)$.
\end{proof}

\begin{rems}\label{essdiffjau}%
\hfil
\smallskip

1.\
Up to holomorphic linear coordinate changes there exist precisely $n$
essentially different ideals $J(a,U)$ in $\C[\z,\ol\z]$. Indeed, we
can assume that $a=0$ and that
$$U\>=\>U_r\>:=\>\diag(1,\dots,1,0,\dots,0)$$
is the diagonal matrix of rank $r$, where $1\le r\le n$ can be
arbitrary. In this case, $J(0,U_r)$ consists of all $f\in
\C[\z,\ol\z]$ which are modulo $(z_1,\dots,z_n)^2+(\ol z_1,\dots,
\ol z_n)^2$ congruent to
$$\sum_{j=r+1}^n(b_jz_j+b'_j\ol z_j)+\sum_{j,k=1}^nc_{jk}\,
z_j\ol z_k$$
with $b_j,\,b'_j,\,c_{jk}\in\C$ and
$$c_{11}+\cdots+c_{rr}\>=\>0.$$
A system of generators for the ideal $J(0,U_r)$ is therefore given by
the following list of polynomials:
$$\begin{array}{ll}
z_jz_k,\ \ol z_j\ol z_k & 1\le j\le k\le r, \\ [2pt]
z_j\ol z_k,\ z_k\ol z_j & 1\le j<k\le r, \\ [2pt]
|z_j|^2-|z_{j+1}|^2 & 1\le j<r, \\ [2pt]
z_j,\ \ol z_j & r+1\le j\le n.
\end{array}$$
\smallskip

2.\
In \cite{PS}, the ideals
$$J_{a,a}\>=\>(z_1-a_1,\dots,z_n-a_n)^2+(\ol z_1-\ol a_1,\dots,
\ol z_n-\ol a_n)^2\quad(a\in\C^n)$$
of $\C[\z,\ol\z]$ were used. They ideals relate to the ideals
$J(a,U)$ studied here via
$$J_{a,a}\>=\>\bigcap_UJ(a,U),$$
intersection over all nonnegative hermitian matrices $U\ne0$. In
particular, in the one variable case ($n=1$) we have $J_{a,a}=
J(a,1)$.
\end{rems}

As a consequence of Theorem \ref{umker} and Proposition \ref{synth},
we obtain:

\begin{cor}\label{idtheorobst}%
Let $I\subset\R[\x,\y]$ be an ideal, and assume that $I\subset
I(a,b)$ for some $a\ne b$ in $\C^n$, or that $I\subset J(a,U)$ for
some $a\in\C^n$ and some nonnegative hermitian $n\times n$ matrix
$U\ne0$. Then there exists $f\in\R[\x,\y]$ such that $f>0$ on
$V_\R(I)$, but $f$ is not a hermitian sum of squares modulo~$I$.
\qed
\end{cor}

In the first case of Corollary \ref{idtheorobst}, the assertion was
already proved in \cite{PS} Proposition 3.1, by a different argument.


\section{Examples}

We start by identifying some classes of (principal) ideals that
satisfy the subnormality condition (S).

\begin{prop}\label{ellgend}%
Let $f=f^*\in\C[\z,\ol\z]$ be of the form
$$f\>=\>\Re g(\z)-\sum_{k=1}^r|q_k(\z)|^2$$
where $g,\,q_1,\dots,q_r\in\C[\z]$. Assume for every $j=1,\dots,n$
that $z_j$ is a polynomial in $g,\,q_1,\dots,q_r$, that is,
$$\C[g,\,q_1,\dots,q_r]\>=\>\C[z_1,\dots,z_n].$$
Then every commuting tuple $T$ satisfying $f(T,T^*)=0$ is subnormal.
In other words, the principal ideal $I=(f)$ has property~$(\rm S)$.
\end{prop}

\begin{proof}
Choose a real number $c>0$ so large that the operator $A:=g(T)
+c\,\id$ is invertible. From $2c\,\Re(g)=|g+c|^2-c^2-|g|^2$ we get
$$2cf\>=\>|g+c|^2-c^2-|g|^2-2c\sum_k|q_k|^2.$$
This implies
$$A^*A\>=\>c^2\,\id+g(T)^*g(T)+2c\sum_kq_k(T)^*q_k(T),$$
hence suitable scalings of the commuting operators
$$A^{-1},\ g(T)A^{-1},\ q_1(T)A^{-1},\>\dots,\>q_r(T)A^{-1}$$
satisfy the identity of the sphere. Therefore the tuple consisting of
these operators is subnormal, by Athavale's theorem \cite{At}. Using
rational functional calculus in conjunction with the spectral
inclusion theorem \cite{P0}, we conclude that the tuple $\bigl(g(T)$,
$q_1(T),\dots,q_k(T)\bigr)$ commutes and is subnormal. Now the
hypothesis implies that the tuple $T=(T_1,\dots,T_n)$ is subnormal.
\end{proof}

\begin{lab}\label{gendlemniscate}%
We discuss yet another class of identities that entail the
subnormality condition (S), this time in one variable ($n=1$). Let
$f\in\R[x,y]$ have the form
$$f\>=\>|g(z)|^2-a-|l(z)|^2-\sum_{k=1}^r|q_k(z)|^2$$
where $a>0$ is a real number, $g,\,q_1,\dots,q_r\in\C[z]$ are
arbitrary polynomials and $l\in\C[z]$ has degree one. The identity
$f(T,T^*)=0$ implies $g(T)^*g(T)\ge aI$, and we conclude that $g(T)$
is invertible. Inverting $g(T)$ we again arrive at a sphere identity,
and arguing as in \ref{ellgend} we conclude in particular that $l(T)$
is subnormal, whence $T$ is subnormal.

This construction can also be performed in any number of variables.
\end{lab}

\begin{lab}\label{exnotarch}%
In certain cases we can prove that $V_\R(f)$ is compact and
$\Sigma_h$ is not archimedean modulo~$f$, for $f$ as in
\ref{gendlemniscate}. Indeed, let
$$f\>=\>|z|^{2m}-\sum_{j=0}^{m-1}a_j|z|^{2j}$$
with $m\ge2$ and real coefficients $a_0,\dots,a_{m-1}\ge0$. Then
$\Sigma_h+(f)$ is not archimedean. Indeed, assume $c-|z|^2+fg\in
\Sigma_h$, with $c\in\R$ and $g=g^*\in\C[z,\ol z]$. Let $b_j$ be the
coefficient of $|z|^{2j}$ in $g$. For any $j\ge0$, the coefficient of
$|z|^{2j}$ in $c-|z|^2+fg$ is $\ge0$. For $j=1$ this gives
\begin{equation}\label{coeff1}%
-1-a_1b_0-a_0b_1\ge0,
\end{equation}
while for $j=m+k$ with $k\ge0$ it gives
\begin{equation}\label{eqbk}%
b_k-a_{m-1}b_{k+1}-\cdots-a_0b_{k+m}\>\ge\>0\quad(k\ge0).
\end{equation}
Let $l\ge0$ be the largest index for which $b_l\ne0$ (by
\eqref{coeff1}, there has to be such~$l$). From \eqref{eqbk} for
$k=l$ we get $b_l>0$. By a downward induction, repeatedly using
\eqref{eqbk}, we conclude that $b_k\ge0$ holds for all $k\ge0$.
But this contradicts \eqref{coeff1}. On the other hand, property
$(2)$ holds as soon as $a_0>0$ and $a_1>0$, see \ref{gendlemniscate}.
The zero set $V_\R(f)$ is the union of (at least one, at most $m-1$)
concentric circles around~$0$.
\end{lab}

Next, we look at the simplest case, which is plane conics. The
following theorem shows that we can completely decide in which cases
the various properties discussed so far are satisfied. In particular,
it turns out that properties (S) and (Sf) are equivalent for plane
conics:

\begin{thm}\label{quadpolonevar}%
Consider a nonconstant polynomial
$$f\>=\>az\ol z+\alpha z^2+\ol\alpha\ol z^2+\beta z+\ol\beta\ol z+c$$
with $a,\,c\in\R$ and $\alpha,\,\beta\in\C$, and let $(f)$ be the
principal ideal generated by $f$ in $\R[x,y]$.
\begin{itemize}
\item[(a)]
$(f)$ has the Archimedean property $(\rm A)$ if, and only if,
$\alpha=0$ and $a\ne0$.
\item[(b)]
$(f)$ has the Quillen property $(\rm Q)$ if, and only if, $\alpha=0$.
\item[(c)]
Properties $(\rm S)$, $(\rm Sf)$ and $(\rm G)$ for the ideal $(f)$
are equivalent among each other, and are also equivalent to ($a\ne0$
\ $\lor$ \ $a=\alpha=0$).
\end{itemize}
\end{thm}

Note that the result \cite{PSell} Theorem 3.3 on the ellipse is
contained in (c) as a particular case.

For the proof of the theorem we need the following simple
observation. (A~similar argument was used in \cite{PSell}, proof of
Proposition 3.1.) The leading form $\lf(f)$ of $0\ne f\in
\C[\z,\ol\z]$ is the nonvanishing homogeneous part of $f$ of highest
degree. Clearly $\lf(fg)=\lf(f)\lf(g)$ and $\lf(f^*)=\lf(f)^*$.

\begin{lem}\label{leitformsohs}%
($n$ arbitrary)
For $0\ne f=f^*\in\Sigma_h$ we have $\lf(f)\in\Sigma_h$. In
particular, when $n=1$, this implies $\lf(f)=a(z\ol z)^m$ where
$\deg(f)=2m$ and $0<a\in\R$.
\end{lem}

The lemma is obvious since in a sum $\sum_j|q_j(\z)|^2$, no
cancellation of leading forms can occur.

\begin{proof}[Proof of Theorem \ref{quadpolonevar}]
Assume $\alpha=0$ and $a\ne0$. Then the identity
$$af\>=\>\bigl|az+\ol\beta\bigr|^2+(ac-|\beta|^2),$$
combined with Proposition \ref{sohsarchmodi}, shows that $\Sigma_h
+(f)$ is archimedean and (hence) contains every polynomial $g$ with
$g>0$ on $V_\R(f)$.
If $\alpha=a=0$ then $f$ is linear, and after a holomorphic change of
variables we may assume $f=\frac1{2i}(z-\ol z)=y$.
By Remark \ref{roleofcpt}.1 we see that $\Sigma_h+(f)$ contains every
$g\in\R[x,y]$ with $g\ge0$ on $V_\R(f)$.

Conversely, we show that $\Sigma_h+(f)$ cannot contain all
polynomials strictly positive on $V_\R(f)$ when $\alpha\ne0$. Indeed,
assume $\alpha\ne0$ and choose $\gamma\in\C$ with $\gamma\notin
V_\R(f)$. For sufficiently small real $r>0$, the polynomial
$g=|z-\gamma|^2-r^2$ is strictly positive on $V_\R(f)$. Assuming
$g\in\Sigma_h+(f)$ would mean $g+fh\in\Sigma_h$ for some $h\in
\C[z,\ol z]$, and necessarily $h\ne0$. When $h$ is constant then
$l(g+fh)$ contains $\lambda z^2$ for some $\lambda\ne0$,
contradicting Lemma \ref{leitformsohs}. Otherwise $\deg(h)>0$, and
then $\lf(f)$ divides $\lf(fh)=\lf(g+fh)$, again contradicting
\ref{leitformsohs}.

We have thus proved (a) and (b). For the proof of (c) we easily
dispense with the linear case $a=\alpha=0$,
and can assume $\deg(f)=2$. If $a=0$ then $f=g+g^*$ with a quadratic
holomorphic polynomial $g\in\C[z]$. For generic choice of $t\in\R$
there are two different numbers $\alpha\ne\beta$ in $\C$ with
$g(\alpha)=g(\beta)=it$. For any such pair we have $f(\alpha)=
f(\beta)=f(\alpha,\ol\beta)=0$, and hence $f\in I(\alpha,\beta)$.
For $a=0$, therefore, the ideal $(f)$ does not satisfy condition (G),
and \emph{a~fortiori} does not satisfy conditions (S) and (Sf), by
\ref{umker}.

On the other hand, assume $a\ne0$. Then $f$ has the form
$$f\>=\>\Re g(z)+a|z|^2$$
with $g\in\C[z]$. According to Proposition \ref{ellgend}, every
bounded operator $T$ satisfying $f(T,T^*)=0$ is subnormal. So $(f)$
satisfies conditions (S), (Sf) and (G) in this case.
\end{proof}

\begin{rems}\label{remsquadpol}%
\hfil
\smallskip

1.\
We rephrase part of Theorem \ref{quadpolonevar} in geometric terms,
and assume $\deg(f)=2$ and $V_\R(f)\ne\emptyset$ for simplicity. Then
$\Sigma_h+(f)$ contains all polynomials positive on $V_\R(f)$ if and
only if $V_\R(f)$ is a circle. On the other hand, the identity
$f(T,T^*)=0$ implies subnormality for a bounded operator $T$ if and
only if $V_\R(f)$ is \emph{not} a hyperbola with perpendicular
asymptotes (and neither a union of two perpendicular lines).
\smallskip

2.\
From Theorem \ref{quadpolonevar} we see in particular that there
exist ideals $I\subset\R[x,y]$ with $V_\R(I)=\emptyset$ for which
$-1\notin\Sigma_h+I$. This is in striking contrast to the case of
usual sums of squares, where it is well known that $V_\R(I)=
\emptyset$ implies $-1\in\Sigma+I$.
Such ideals may well have the subnormality property (S). For example,
this is so for $I=(ax^2+by^2+c)$ with $a,\,b,\,c>0$ and $a\ne b$.
\end{rems}

\begin{lab}\label{pu27jun}%
For (reduced) plane conics, the normality condition (Sf) for
finite-dimensional Hilbert spaces already implies the subnormality
condition (S) for arbitrary Hilbert spaces, as shown in
\ref{quadpolonevar}. We now show that this ceases to hold when we
take a suitable nonreduced version of a conic.

To this end consider the $*$-invariant ideal
$$J\>=\>\Bigl((z-1)(z\ol z-1),\ (\ol z-1)(z\ol z-1)\Bigr)$$
in $\C[z,\ol z]$, respectively its real version
$$I\>=\>J\cap\R[x,y]\>=\>(x^2+y^2-1)\cdot(x-1,y).$$
The ideal corresponds to the unit circle with nilpotents added at one
point.
We will see that hermitian sums of squares modulo~$I$ behave quite
different than modulo~$\sqrt I$.

Let $T\in B(E)$ satisfy $I\subset\ker(\psi_T)$, that is,
\begin{equation}\label{eqt}%
(T^*-\id)(T^*T-\id)\>=\>0.
\end{equation}
We decompose the Hilbert space as $E=\ker(T-\id)\oplus\ker(T-\id)
^\bot$. With respect to this decomposition, $T$ has a block matrix
representation
$$T\>=\>\begin{pmatrix}\id&A\\0&B\end{pmatrix}$$
with $\ker(A)\cap\ker(B-\id)=\{0\}$.
From \eqref{eqt} we deduce $A(B-\id)=0$ and $A^*A+(B^*B-\id)(B-\id)
=0$.
The second identity implies that $B-\id$ is actually injective.
If $\dim(E)<\infty$, then $B-\id$ is invertible, and we get $A=0$ and
$B^*B-\id=0$. In short, $T$ is unitary. Every (finite-dimensional)
matrix annihilated by the ideal $J$ is therefore unitary, and hence
normal.

On the other hand, we will produce an operator $T$ acting on $E=
\ell^2(\N)$ such that $T$ is annihilated by the ideal $J$ and $T$ is
not subnormal. Let $E$ have Hilbert basis $e_k$ ($k\ge0$), and let
$S\colon e_k\mapsto e_{k+1}$ ($k\ge0$) be the unilateral shift. Let
$\pi$ be the orthogonal projection onto the space generated by $e_0$,
and define $T=S+\pi$. A direct computation, supported by the
relations
$$S^*S=\id,\quad SS^*=\id-\pi,\quad \pi S=0,$$
yields $(T^*-\id)(T^*T-\id)=0$.
But $T$ is not subnormal. Indeed, any subnormal operator $X$
satisfies the hyponormality inequality $[X^*,X]\ge0$, evident from
the matrix decomposition of a normal extension $N=\choose{X\ Y}
{0\ Z}$ and the equation $[N^*,N]=0$.
The commutator $[T^*,T]=\pi-S\pi-\pi S^*$ acts on the span of $e_0$
and $e_1$ as $\begin{pmatrix}1&-1\\-1&0\end{pmatrix}$, and this
is not a nonnegative operator. In summary, this shows:
\end{lab}

\begin{prop}\label{sfnots}%
Let $I\subset\R[x,y]$ be the above ideal, corresponding to the unit
circle with a thickened point. Then $I$ satisfies condition
$(\rm Sf)$, but not~$(\rm S)$ (and a~fortiori, not~$(\rm A)$). In
contrast, its reduced version $\sqrt I$ satisfies $(\rm A)$ (and
therefore also $(\rm S)$ and $(\rm Sf)$).
\qed
\end{prop}


\section{Semi-algebraic sets}

\begin{lab}\label{subtle}%
In Section \ref{sect:hermsossubnop} we studied the question whether
every polynomial strictly positive on a real algebraic set
$X\subset\C^n$ is a hermitian sum of squares on $X$. We now extend
this question to (real) semi-algebraic subsets of $\C^n$.
Algebraically, this means that instead of an ideal $I\subset
\R[\x,\y]$ and the semiring $\Sigma_h+I$ we consider \emph{hermitian
modules}, that is, modules over the semiring $\Sigma_h$ (see
\ref{dfnarch}). This means that the real algebraic set $X=V_\R(I)$
is replaced by the closed set
$$X_M\>=\>\{a\in\C^n\colon\all f\in M\ f(a)\ge0\}$$
(see \ref{archpss}). If the hermitian module $M$ is finitely
generated (or, more generally, if the quadratic module generated by
$M$ is finitely generated), the closed set $X_M$ is basic closed,
i.e., there are finitely many $f_1,\dots,f_k\in\R[\x,\y]$ with $X_M=
\{a\in\C^n$: $f_1(a)\ge0,\dots,f_k(a)\ge0\}$.
\end{lab}

\begin{lab}
Each of the four properties of an ideal $I\subset\R[\x,\y]$ labelled
(A), (Q), (S), (Sf) that were discussed in the first part of this
paper is in fact a property of the semiring $S=\Sigma_h+I$, i.e., can
be expressed in terms of $S$. We now extend these properties to
arbitrary hermitian modules $M\subset\R[\x,\y]$:
\begin{itemize}
\item[(A)]
$M$ is archimedean;
\item[(Q)]
$M$ contains every $f\in\R[\x,\y]$ with $f>0$ on $X_M$;
\item[(S)]
every commuting tuple $T$ of bounded operators in a Hilbert space
satisfying $p(T,T^*)\ge0$ for every $p\in M$ is subnormal;
\item[(Sf)]
every commuting tuple $T$ of complex matrices satisfying $p(T,T^*)
\ge0$ for every $p\in M$ is normal.
\end{itemize}
When $M=\Sigma_h+I$ for some ideal $I$, the above properties agree
with the respective properties of the ideal~$I$, as defined in
\ref{dfnpropa} and \ref{dfnssf}.
\end{lab}

We start by the following characterization of archimedean hermitian
modules, thereby generalizing part of Proposition \ref{sohsarchmodi}:

\begin{lem}\label{archhermmod}%
A hermitian module $M\subset\R[\x,\y]$ is archimedean if and only if
$c-||\z||^2\in M$ for some real number~$c$.
\end{lem}

\begin{proof}
Assuming $c-||\z||^2\in M$ we have to show $\R+M=\R[\x,\y]$. For
this let $A:=\{p\in\C[\z]\colon-|p|^2\in\R+M\}$. It suffices to prove
$A=\C[\z]$. Indeed, any $f\in\R[\x,\y]$ can be written $f=\sum_j
|p_j|^2-\sum_k|q_k|^2$ with $p_j,\,q_k\in\C[\z]$; if $q_k\in A$ for
every $k$, then $f\in\R+M$.

From $c-|z_j|^2=(c-||\z||^2)+\sum_{k\ne j}|z_k|^2$ we see $z_j\in A$
for $j=1,\dots,n$. Therefore (and since $\C\subset A$) it is enough
to prove that $A$ is a ring. From $a-|f|^2$, $b-|g|^2\in M$ with
$a,\,b\ge0$ we get
$$ab-|fg|^2\>=a(b-|g|^2)+|g|^2(a-|f|^2)\>\in M,$$
so $A$ is closed under products. From $|f+g|^2+|f-g|^2=2(|f|^2+
|g|^2)$ we see that $A$ is also closed under sums. The lemma is
proved.
\end{proof}

Before we start discussing a Positivstellensatz for hermitian
modules, we need to mention a subtle point. The archimedean
Positivstellensatz \ref{archpss} holds for modules over archimedean
semirings, but \emph{not} in general for archimedean modules over
semi\-rings. This distinction is relevant for hermitian modules, as
the following example shows:

\begin{example}\label{pssarchhermmod}%
The hermitian module $M=\Sigma_h+\Sigma_h(1-||\z||^2)$ is
archimedean by Lemma \ref{archhermmod}. But there exist polynomials
that are strictly positive on the closed unit ball $X_M$ and are not
contained in $M$. In fact, $\epsilon+(1-||\z||^2)^2$ is such a
polynomial for $0<\epsilon<1$. To see this, assume
$$\epsilon+(1-||\z||^2)^2\>=\>p+q(1-||\z||^2)$$
with $p$, $q\in\Sigma_h$. Comparing constant coefficients gives
$q(0)\le1+\epsilon$, while comparing coefficients of $z_1\ol z_1$
gives $-2\ge-q(0)$, i.e.\ $q(0)\ge2$, since the coefficient of
$z_1\ol z_1$ in any hermitian sum of squares is nonnegative.

The point is that, although the hermitian module $M$ is archimedean,
$M$ is not a module over any archimedean semiring.

The proper ``quantization'' of the module $M$ is a linear operator
$T$ acting on a Hilbert space, subject to the contractivity condition
$$I-T^*T\>\ge\>0.$$
It is clear that not every contractive operator $T$ is subnormal.
\end{example}

Example \ref{pssarchhermmod} has shown (c.f.\ also Proposition
\ref{sohsarchmodi}):

\begin{lem}
Consider the following two properties of a hermitian module $M$ in
$\R[\x,\y]$:
\begin{itemize}
\item[(i)]
$X_M$ is compact, and $M$ has the Quillen property $(\rm Q)$;
\item[(ii)]
$M$ is archimedean $(\rm A)$.
\end{itemize}
Then (i) implies (ii), but the converse fails in general.
\qed
\end{lem}

Here we are mainly interested in a Positivstellensatz, that is, in
the Quillen property (Q), in the case when $X_M$ is compact.
Therefore, we will often \emph{assume} that $M$ is archimedean (which
implies that $X_M$ is compact), and try to find additional properties
for $M$ that will imply the Positivstellensatz. Verifying the
archimedean property of a concretely given hermitian module $M$ is
usually easy, using the criterion of Lemma \ref{archhermmod}.

One instance where we get the Positivstellensatz for free is the
following:

\begin{prop}\label{vaquillthm}%
Let $I\subset\R[\x,\y]$ be any ideal with the archimedean property
$(\rm A)$. Then for any $p_1,\dots,p_r\in\R[\x,\y]$, the hermitian
module
$$M\>=\>I+\Sigma_h+p_1\Sigma_h+\cdots+p_r\Sigma_h$$
has the Quillen property $(\rm Q)$.
\end{prop}

\begin{proof}
$M$ is a module over the archimedean semiring $\Sigma_h+I$, so the
assertion follows from the archimedean Positivstellensatz
\ref{archpss}.
\end{proof}

\begin{rems}
\hfil
\smallskip

1.\
For $M$ as in \ref{vaquillthm}, the associated semi-algebraic set is
$$X_M\>=\>V_\R(I)\cap\Bigl\{a\in\C^n\colon\>p_1(a)\ge0,\>\dots,\>
p_r(a)\ge0\Bigr\}.$$

2.\
In the particular case $I=(1-||\z||^2)$, Proposition \ref{vaquillthm}
was proved by D'Angelo and Putinar (\cite{dAP}, Thm.\ 3.1).
\end{rems}

Generalizing Proposition \ref{synth}, the Positivstellensatz for $M$
implies the following subnormality property for commuting tuples $T$
of bounded operators.

\begin{cor}\label{modules}%
For any hermitian module $M\subset\R[\x,\y]$, we have $(\rm Q)\To
(\rm S)$.
\end{cor}

\begin{proof}
The proof of Proposition \ref{synth} carries over (essentially)
verbatim.
\end{proof}

We next discuss a refinement of the last corollary, in which we are
going to weaken condition (Q) and strengthen condition (S).

\begin{lab}\label{dfnsmp}%
Let $M\subset\R[\x,\y]$ be a hermitian module. Recall \cite{Sm} that
$M$ is said to have the \emph{Strong Moment Property} if the
following holds:
\begin{itemize}
\item[(\rm SMP)]
Every linear functional $L\colon\R[\x,\y]\to\R$ with $L|_M\ge0$ is
integration with respect to some positive Borel measure on $X_M$.
\end{itemize}
\end{lab}

\begin{lem}\label{2.2}%
Let $M\subset\R[\x,\y]$ be a hermitian module. The Quillen property
$(\rm Q)$ for $M$ implies the strong moment property $(\rm SMP)$
for~$M$. The converse is true if $M$ is archimedean.
\end{lem}

\begin{proof}
Assume that $M$ has property (Q). Then any $f\in\R[\x,\y]$
nonnegative on $X_M$ lies in the closure $\ol M$ with respect to the
finest locally convex topology on $\R[\x,\y]$.
So $L|_M\ge0$ implies $L(f)\ge0$, and therefore $L$ is integration
with respect to some positive Borel measure on $X_M$, according to
the Riesz-Haviland theorem \cite{Hv}.
Conversely, let $M$ be archimedean and satisfy (SMP), and assume that
there is $f\in\R[\x,\y]$ with $f>0$ on $X_M$ but $f\notin M$. By
Eidelheit's separation theorem (e.g., \cite{Ko} \S17.1) there is a
linear functional $L\colon\R[\x,\y]\to\R$ with $L(1)=1$, $L|_M\ge0$
and $L(f)\le0$.
By assumption, $L$ is integration with respect to a probability
measure $\mu$ on $X_M$. We conclude $L(f)=\int_{X_M}f\>d\mu>0$, a
contradiction.
\end{proof}

\begin{lem}\label{surpris}%
For $M$ a hermitian module in $\R[\x,\y]$, consider the following
property:
\begin{itemize}
\item[$(\rm SOS)$]
$\Sigma\subset\ol M$ (closure with respect to the finest locally
convex topology on $\R[\x,\y]$).
\end{itemize}
Then property $(\rm SOS)$ for $M$ implies property $(\rm S)$ for $M$.
The converse is true if $M$ is archimedean.
\end{lem}

\begin{proof}
(SOS) $\To$ (S):
Let $T$ be a commuting tuple of operators with $M\subset M_T$. Since
$M_T$ is closed in $\R[\x,\y]$ (\ref{heredcalc}) we have $\ol M
\subset M_T$. So assumption (SOS) implies $\Sigma\subset M_T$, and
hence $T$ is subnormal according to Lemma \ref{sigmasubn}.

(S) $\To$ (SOS):
Let $M$ be archimedean and have property (S). Assume that there is
$f\in\Sigma$ with $f\notin\ol M$. There exists a linear functional
$L\colon\R[\x,\y]\to\R$ satisfying $L|_M\ge0$ and $L(f)<0$.
We extend $L$ to a complex linear functional on $\R[\x,\y]\otimes_\R
\C=\C[\z,\ol\z]$ and perform a GNS construction: Consider the
positive semi-definite inner product $\bil pq:=L(pq^*)$ on $\C[\z]$,
and let $E$ be the corresponding Hilbert space completion of
$\C[\z]$. Since $M$ is archimedean, there is a real number $c>0$ with
$c-|z_j|^2\in M$ for $j=1,\dots,n$, and it follows for any $p\in
\C[\z]$ that
$$L(|z_j|^2\,|p(\z)|^2)\>\le\>c\cdot L(|p(\z)|^2).$$
Hence multiplication by $z_j$ induces a bounded linear operator $T_j$
on $E$ ($j=1,\dots,n$), and $T=(T_1,\dots,T_n)$ is a commuting tuple
in $B(E)$. For $g\in\C[\z,\ol\z]$ and $q\in\C[\z]$ we have $\langle
\psi_T(g)q,\,q\rangle=L(g|q|^2)$.
So for $g\in M$ the operator $\psi_T(g)$ is nonnegative, which means
$M\subset M_T$. On the other hand, $f\notin M_T$ since $\bil
{\psi_T(f)1}1=L(f)<0$. Therefore, the tuple $T$ is not subnormal,
according to Proposition \ref{sigmasubn}. This contradicts
property~(S).
\end{proof}

\begin{lem}\label{1aTo2a}%
For any hermitian module $M$ we have $(\rm SMP)$ $\To$ $(\rm SOS)$.
\end{lem}

\begin{proof}
By hypothesis, any linear functional $L\colon\R[\x,\y]\to\R$ with
$L|_M\ge0$ is integration with respect to a measure on $X_M$. In
particular, $L(f)\ge0$ for any $f\in\Sigma$. This implies $f\in
\ol M$.
\end{proof}

\begin{rems}
\hfil\smallskip

1.\
The implication (S) $\To$ (SOS) for archimedean $M$ (Lemma
\ref{surpris}) is uninteresting if $M$ is a module over an
archimedean semiring. Indeed, in this case we know anyway that $M$
contains all polynomials strictly positive on $X_M$,
and hence $\Sigma\subset\ol M$ is clear. But in the other cases,
the equivalence of (S) and (SOS) is a new information.
\smallskip

2.\
Altogether we have now obtained the chain of implications
$$(\rm Q)\ \To\ (\rm SMP)\ \To\ (\rm SOS)\ \To\ (\rm S)$$
for any hermitian module $M$. When $M$ is archimedean, the first and
the last implication can be reversed.
\end{rems}

\begin{lab}
Under a stronger condition on $M$, we are now going to prove the
implication (SOS) $\To$ (SMP), and hence the equivalence of (Q)
and (S), when $M$ is archimedean. Recall that a closed subset
$K\subset\C^n$ is said to be \emph{polynomially convex} if the
following holds: For every $\xi\in\C^n$ with $\xi\notin K$, there
exists a polynomial $p\in\C[\z]$ with $|p|\le1$ on $K$ and $|p(\xi)|
>1$. We shall consider the following property for a hermitian module
$M\subset\R[\x,\y]$:
\begin{itemize}
\item[(PC)]
For every $\xi\in\C^n\setminus X_M$, there exist $f\in M$ and $q\in
\Sigma_h$ such that $q\le1$ on $\{a\in\C^n\colon f(a)\ge0\}$, and
such that $q(\xi)>1$.
\end{itemize}
\end{lab}

\begin{rems}
\hfil
\smallskip

1.\
If $M$ satisfies condition $(\rm PC)$, then the set $X_M$ is
polynomially convex in $\C^n$. Indeed, $X_M$ has the form
$$X_M\>=\>\bigcap_\nu\{a\in\C^n\colon q_\nu(a)\le1\}$$
for some family of polynomials $q_\nu\in\Sigma_h$. For each $\nu$,
the set $\{a\in\C^n\colon q_\nu(a)\le1\}$ is polynomially convex,
since it is the preimage of the closed unit ball in some $\C^m$
under a polynomial map $\C^n\to\C^m$. Therefore $X_M$ is
polynomially convex.
\smallskip

2.\
Condition (PC) is satisfied when the hermitian module $M$ is
generated by polynomials of the form $1-q_\nu$ with $q_\nu\in
\Sigma_h$.
More generally, (PC) holds when $M$ contains a family $\{f_\nu\}$
of polynomials with $X_M=\bigcap_\nu\{a\in\C^n\colon f_\nu(a)\ge0\}$
such that, for every $\nu$, the set $\{f_\nu\ge0\}$ is polynomially
convex.
\end{rems}

\begin{thm}\label{conversepolyconv}%
Let $M$ be an archimedean hermitian module in $\R[\x,\y]$ which
satisfies condition $(\rm PC)$. Then the subnormality property
$(\rm S)$ implies the Quillen property $(\rm Q)$ for~$M$.
\end{thm}

\begin{proof}
By Lemma \ref{2.2} it suffices to show that $M$ has the strong moment
property (SMP). Given a linear functional $L\colon\R[\x,\y]\to\R$
with $L|_M\ge0$, we need to show that $L$ is integration with respect
to a positive Borel measure supported on $X_M$. Similar to the proof
of Lemma \ref{surpris}, we use a GNS construction to get a Hilbert
space $E$ together with a commuting tuple $T=(T_1,\dots,T_n)$ in
$B(E)$ and a cyclic vector $\xi$, such that $L(pq^*)=\bil{p(T)\xi}
{q(T)\xi}$ for all $p,q\in\C[\z]$.
Since $M$ has property (S), the tuple $T$ is subnormal.
The spectral
measure of a commuting normal extension $S$ of $T$ gives a Borel
measure $\mu$ on $\C^n$ with
$$L(f)\>=\>\bil{\psi_S(f)\xi}\xi\>=\>\int f\>d\mu$$
for all $f\in\R[\x,\y]$.
It remains to prove $\supp(\mu)\subset X_M$, which follows from the
next lemma.
\end{proof}

\begin{lem}\label{supplemm}%
Let $M$ be a hermitian module, and let $\mu$ be a
positive measure on $\C^n$ all of whose moments exist, satisfying
$\int f\,d\mu\ge0$ for every $f\in M$. If $M$ satisfies condition
$(\rm PC)$, then $\supp(\mu)\subset X_M$.
\end{lem}

\begin{proof}
Assume there exists $\xi\in\supp(\mu)$ with $\xi\notin X_M$. By
(PC) we find $f\in M$ and $q\in\Sigma_h$ such that $q(\xi)>1$ and
$q<1$ on $\{f\ge0\}$. We have
$$\int fq^m\,d\mu\>=\>\int_{\{f\ge0\}}fq^m\,d\mu+\int_{\{f<0\}}fq^m\,
d\mu.$$
For $m\to\infty$, the first summand on the right tends to zero by the
dominated convergence theorem.
On the other hand, the second summand tends to $-\infty$: Consider a
small ball $B$ around $\xi$ on which $q\ge a>1$ and $f\le b<0$, and
note that $B$ has positive $\mu$-measure.
So there is $m\in\N$ for which the integral on the left is negative,
a contradiction since $fq^m\in M$.
\end{proof}

Summarizing Lemmas \ref{2.2}--\ref{1aTo2a} and Proposition
\ref{conversepolyconv}, we obtain:

\begin{thm}
Let $M$ be a hermitian module in $\R[\x,\y]$ which is archimedean
and satisfies condition $(\rm PC)$. Then for $M$ we have
$$(\rm Q)\ \iff\ (\rm SMP)\ \iff\ (\rm SOS)\ \iff\ (\rm S).
\eqno\square$$
\end{thm}

\begin{example}\label{putex}%
Without any hypothesis like (PC), the implication (S) $\To$ (Q) is
false, even if we assume that $M$ is archimedean. Indeed, let $n=1$
and $0<r<R$, and consider the $\Sigma_h$-module
$$M\>=\>\Sigma+\Sigma_h(R^2-|z|^2)+\Sigma_h(|z|^2-r^2).$$
Here $X_M$ is the annulus around the origin with radii $r<R$.
Clearly, $M$ is archimedean (\ref{archhermmod}) and satisfies
condition (S) (Lemma \ref{surpris}).
But there exists a compactly supported measure $\mu$ on $\C$ with
$\supp(\mu)\not\subset X_M$ and with
$$\int f(z)\,\mu(dz)\>\ge\>0$$
for every $f\in M$.
Namely, let $r<\rho<R$, and let
$$\int f\,d\mu\>:=\>\epsilon f(0)+\int_{-\pi}^\pi f(\rho e^{it})
\,dt.$$
When $\epsilon>0$ is sufficiently small, we have $\int(|z|^2-r^2)
|p(z)|^2\,d\mu\ge0$ for every $p\in\C[z]$, and hence $\int f(z)\,
\mu(dz)\ge0$ for every $f\in M$. Namely, the integral is
$$(\rho^2-r^2)\int_{-\pi}^\pi|p(\rho e^{it})|^2\,dt-\epsilon r^2
|p(0)|^2\>\ge\>\bigl(2\pi(\rho^2-r^2)-\epsilon r^2\bigr)|p(0)|^2.$$
Note that the annulus $X_M$ is not polynomially convex, and so the
hypotheses of Theorem \ref{conversepolyconv} are not satisfied.
\end{example}

\begin{thm}\label{ausbau}%
Let $M\subset\R[\x,\y]$ be a finitely generated hermitian module. If
$M$ satisfies $(\rm Sf)$, then the semi-algebraic set $X_M$ in $\C^n$
does not contain an analytic disc.
\end{thm}

\begin{proof}
By an analytic disc we mean the image of a nonconstant holomorphic
map $\phi\colon\D\to\C^n$. We can assume $\phi(0)=0$, and since we
work locally, we can assume that there exists an analytic function
$F\colon\phi(\D)\to\D$ such that $F(\phi(\zeta))=\zeta$ ($\zeta
\in\D$). By assumption, $\phi(\D)\subset X_M$. Let $M$ be generated
by nonzero polynomials $f_1,\dots,f_r\in\R[\x,\y]$. We reorder the
generators so that $f_1,\dots,f_s$ vanish identically on $\phi(\D)$,
and $f_{s+1},\dots,f_r$ have only isolated zeros on $\phi(\D)$. By
passing to an appropriate subdisc on $\D$ we can assume $f_{s+1}(0)
>0,\dots,f_r(0)>0$.

Let $\epsilon>0$ be sufficiently small, and choose a non-normal
matrix $A$ of norm $||A||\le\epsilon$. Then the commuting tuple of
matrices $\phi(A)$ is not normal, as $A=F(\phi(A))$ by the
superposition property of the analytic functional calculus. In
addition, $M\subset M_{\phi(A)}$. Indeed, fixing $1\le j\le s$, the
composite function $f_j(\varphi(z),\ol{\varphi(z)})$ is identically
zero. If
$$f_j(\z,\ol\z)\>=\>\sum_{\alpha,\beta}c_{\alpha,\beta}\,
\z^\alpha\ol\z^\beta,$$
this means that the power series
$$\sum_{\alpha,\beta}c_{\alpha,\beta}\,\varphi_1(z)^{\alpha_1}\cdots
\varphi_n(z)^{\alpha_n}\,\ol{\varphi_1(z)}^{\beta_1}\cdots
\ol{\varphi_n(z)}^{\beta_n}$$
in $z$ und $\ol z$ is identically zero, from which we see $f_j
(\phi(A))=0$.
On the other hand, $f_{s+1}(\phi(A))>0,\dots,f_r(\phi(A))>0$ by the
continuity of the functional calculus.

Since $\phi(A)$ is not (sub-) normal, this implies that $M$ does
not satisfy~(Sf).
\end{proof}

\begin{rem}\label{remausbauideal}%
In the case where $M=\Sigma_h+I$ for some ideal $I\subset\R[\x,\y]$,
Theorem \ref{ausbau} also follows from Theorem \ref{umker}. Indeed,
assume that $\varphi\colon\D\to\C^n$ is a holomorphic map with
$\varphi(\D)\subset V_\R(I)$ and with $u:=\nabla_z\varphi(0)\ne0$. A
direct calculation shows that $I\subset J(a,U)$ (see \ref{dfnjau})
with $a:=\varphi(0)$ and $U:=u^*u$ (we consider $u$ as a row vector).
Therefore, the ideal $I$ does not satisfy condition (G), and by
Theorem \ref{umker}, it neither satisfies (Sf).
\end{rem}

By Corollary \ref{modules}, Theorem \ref{ausbau} implies:

\begin{cor}
Let $M$ be a finitely generated hermitian module. If $X_M$ contains
an analytic disc, then $M$ does not satisfy Quillen's property
$(\rm Q)$.
\qed
\end{cor}

\begin{proof}
Let $M$ be generated by nonzero polynomials $f_1,\dots,f_r\in
\R[\x,\y]$. Since $X_M$ contains an analytic disc, the subnormality
condition (S) does not hold for $M$, by  Theorem \ref{ausbau} above.
According to Corollary \ref{modules}, the Positivstellensatz does not
hold either.
\end{proof}


\section{Historical comments}

We feel that leaving aside the analytic roots of the questions
encountered in this article would deprave the reader of some
essential insight. We briefly describe below some of the old sources
and applications of the decomposition of a real polynomial in a sum
of hermitian squares.

Start with Riesz-Fej\'er Theorem asserting that a polynomial
$p(z,\ol z)$ which is non-negative on the unit circle can be
decomposed as
\begin{equation}\label{1d-torus}
p(z,\ol z)\>=\>|h(z)|^2+(1-|z|^2)g(z,\ol z),
\end{equation}
where $h\in\C[z]$ and $g\in\C[z,\ol z]$.

Next we ``quantize'' the above setting, that is we replace the
complex variable by a linear transformation. Let $T$ be a bounded
linear operator acting on a Hilbert space $E$ and denote by $T^*$ its
adjoint. The simple operator identity
$$T^*T\>=\>\id$$
defines an isometric transformation, with the known consequences:
spectral picture, functional model and classification, see \cite{Co}.
In particular, the operator $T$ has in this case spectrum contained
in the closed unit disc $\ol\D$, and for every real valued polynomial
$p(z,\ol z)$ the estimate
\begin{equation}\label{harmonic}
p(T,T^*)\>\le\>\max_{\lambda\in\T}\,p(\lambda,\ol\lambda)\,\id
\end{equation}
holds true. Recall that here we adopt the hereditary calculus
convention, putting the powers of $T^*$ to the left of the powers
of $T$, in every monomial appearing in $p$.

Inequality \eqref{harmonic} is a simple consequence of
\eqref{1d-torus}: If $p(z,\ol z)\le M$ on $\T$, then
\begin{equation}\label{eq}%
M-p(T,T^*)\>=\>h(T)^*h(T)+[(1-|z|^2)g(z,\ol z)](T,T^*)\>=\>
h(T)^*h(T)\>\ge\>0.
\end{equation}
As a matter of fact, estimate \eqref{harmonic} implies that the
linear functional calculus $p(z,\ol z)\mapsto p(T,T^*)$ possesses an
additional positivity property. The latter implies, essentially
repeating F.~Riesz construction of the representing measure for a
positive functional, that the operator $T$ is subnormal, that is,
there exists a larger Hilbert space $E\subset K$ and a normal
operator $U$ acting on $K$, such that $U(E)\subset E$ and $U|_E=T$.
By choosing $U$ minimal with this property we can also assume that
the spectrum of $U$ is contained in the torus $\T$, hence $U$ is
unitary, see for instance \cite{Co}. In particular, if in addition
$TT^*=\id$, that is $T$ is unitary from the beginning, we obtain in
this manner a proof of the spectral theorem, as advocated by F.~Riesz
from the dawn of functional analysis \cite{R,RN}.

Turning now to several complex variables, or their quantized form,
commuting tuples of linear operators, we encounter Quillen's idea
\cite{Q}. Let $P(z,\ol z)$ be a conjugation-invariant polynomial,
bihomogeneous of the same degree in the variables $z$ and $\ol z$.
Assume that $P(z,\ol z)>0$ whenever $z\ne0$. Denote by $M=(M_{z_1},
\dots,M_{z_n})$ the $n$-tuple of commuting multipliers by the
complex variables, on the Bargmann-Fock space of entire functions
(square integrable in $\C^n$ with respect to the Gaussian weight).
Using analytical tools (elliptic estimates and Fredholm theory),
Quillen analyzes the positivity of the operator $P(M,M^*)$ inherited
from the positivity of the symbol $P$. He reaches the purely
algebraic conclusion that there exists a positive integer $N$ and
homogeneous complex analytic polynomials $h_1,\dots,h_k$ such that
\begin{equation}\label{homogeneous}
\|z\|^{2N}P(z,\ol z)\>=\>|h_1(z)|^2+\cdots+|h_k(z)|^2.
\end{equation}
Very recently Drout and Zworski \cite{DZ} have obtained, using the
same Bargmann-Fock space representation, degree bounds in Quillen's
decomposition above.

An elementary dehomogenization argument shows that
\eqref{homogeneous} implies that every positive polynomial on the
unit sphere of $\C^n$ is equal, on the sphere, to a sum of hermitian
squares, as stated by condition (Q) in our article.

On the abstract operator theory side, we mention the 1987 discovery
of Athavale \cite{At} stating that every commuting tuple of bounded
operators $T=(T_1,\dots,T_n)$ subject to the sphere identity
$$T_1^*T_1+\cdots+T_n^*T_n\>=\>\id$$
is subnormal, and hence possesses a functional calculus with a
positivity property of type \eqref{harmonic}. Athavale's work belongs
to a framework advocated for several dozen years by now by Conway
\cite{Co}, Agler and McCarthy \cite{AM} and their followers.

Quillen's theorem was rediscovered in 1996, generalized and put into
the context of Cauchy-Riemann geometry and function theory of several
complex variables by Catlin and D'Angelo \cite{CdA1}. Their proof
also uses analysis, this time employing analytic Toeplitz operators
acting on the Bergman space of the unit ball. One of the main themes
of research in Cauchy-Riemann geometry is the (local) classification
up to bi-holomorphic transformations of real algebraic subvarieties
of $\C^n$. There is no surprise that Quillen property, or better its
algebro-geometric consequences (Sf) and (G) are relevant for CR
manifold theory. A modest step into this direction was taken in
\cite{dAP2}.


\end{document}